\documentclass[12pt]{amsart}
\usepackage{amssymb}
\usepackage{amsthm}
\usepackage{verbatim}

\makeatletter
\@namedef{subjclassname@2010}{%
  \textup{2010} Mathematics Subject Classification}
\makeatother

\newtheorem*{fact}{Fact}

\newtheorem*{question}{Question}

\newtheorem{lemma}{Lemma}

\newtheorem*{class}{Classification Theorem}
\newtheorem*{3lemma}{Fixed Vertices Lemma}
\newtheorem*{12lemma}{Orientation Reversing Lemma}
\newtheorem*{0lemma}{Orientation Preserving Lemma}
\newtheorem*{rotation}{Rotations Lemma}
\newtheorem*{glide}{Glide Rotations Lemma}
\newtheorem*{reflection}{Reflections Lemma}
\newtheorem*{improper}{Improper Rotations Lemma}
\newtheorem*{smith}{Smith Theory}
\newtheorem*{autothm}{Automorphism Theorem}

\newtheorem*{eel}{Edge Embedding Lemma}

\def\Z{{\mathbb Z}}

\def\R{{\mathbb R}}

\def\g{{\gamma}}

\def\f{{\varphi}}

\def\fix{{\mathrm{fix}}}

\def\lcm{{\mathrm{lcm}}}

\begin{document}

%%%%% To ease editing, for IMPAN journals add:

\baselineskip=17pt

%%%%%%%%%%%

\title[Bipartite symmetries]{Symmetries of Embedded Complete Bipartite Graphs}
\author{Erica Flapan}
\address{Department of Mathematics, Pomona College, Claremont, CA 91711, USA}
\email{eflapan@pomona.edu}

\author{Nicole Lehle}
\address{Department of Mathematics, Pomona College, Claremont, CA 91711, USA}

\author{Blake Mellor}
\address{Department of Mathematics, Loyola Marymount University, Los Angeles, CA 90045, USA}
\email{blake.mellor@lmu.edu}

\author{Matt Pittluck}
\address{Department of Mathematics, Loyola Marymount University, Los Angeles, CA 90045, USA}

\author{Xan Vongsathorn}
\address{Department of Mathematics, Pomona College, Claremont, CA 91711, USA}

\date{}

\thanks{This research was supported in part by NSF grant DMS-0905687.}

\begin{abstract}

We characterize which automorphisms of an arbitrary complete bipartite graph $K_{n,m}$ can be induced by a homeomorphism of some embedding of the graph in $S^3$.

\end{abstract}



\maketitle

\section{Introduction}
\label{intro}
Knowing the symmetries of a molecule helps to predict its chemical behavior.  Chemists use the {\it point group} as a way to represent the rigid symmetries of a molecule. However, molecules which are large enough to be flexible (such as long polymers) or have pieces which rotate separately may have symmetries which are not induced by rigid motions.  By modeling a molecule as a graph $\Gamma$ embedded in $S^3$, the rigid and non-rigid symmetries of the molecule can be represented by automorphisms of $\Gamma$ which are induced by homeomorphisms of the pair $(S^3, \Gamma)$. 
Different embeddings of the same abstract graph may have different automorphisms which are induced by homeomorphisms of the graph in $S^3$.  In fact, a given automorphism of an abstract graph may or may not be induced by a homeomorphism of some embedding of the graph.  In particular, it was shown in \cite{Fl} that a cyclic permutation of four vertices of the complete graph on six vertices $K_6$ cannot be induced by a homeomorphism of $S^3$, no matter how the graph is embedded in $S^3$.

In general, we are interested in which automorphisms of a graph can be induced by a homeomorphism of some embedding of the graph in $S^3$.  Flapan \cite{Flapan:1995} answered this question for the family of complete graphs $K_n$.  Now we do the same for the family of complete bipartite graphs $K_{n,m}$.  This is an interesting family of graphs to consider because it was shown in \cite{TSG1} that for every finite subgroup $G$ of $\mathrm{Diff}_+(S^3)$, there is an embedding $\Gamma$ of some complete bipartite graph such that the group of all automorphisms of $\Gamma$ which are induced by orientation preserving homeomorphisms of $S^3$ is isomorphic to $G$.  By contrast, it was shown in \cite{TSG2} that this is not the case for the complete graphs.

We prove the following Classification Theorem which determines precisely which automorphisms of a complete bipartite graph can be induced by a homeomorphism of $(S^3, \Gamma)$ for some embedding $\Gamma$ of the graph in $S^3$.  The theorem is divided into two parts, according to whether the homeomorphism inducing a particular automorphism is orientation preserving or reversing.  The automorphisms are described by their fixed vertices and their cycle structure.  Note that in this paper we use the term ``cycle" to refer to a cycle in the permutation on the vertices of a graph induced by an automorphism of the graph, not to a cycle in the graph in the usual graph-theoretic sense.
\medskip

\begin{class} 
Let $m, n>2$ and let $\varphi$ be an order $r$ automorphism of a complete bipartite graph $K_{n,m}$ with vertex sets $V$ and $W$.  There is an embedding $\Gamma$ of $K_{n,m}$ in $S^3$ with an orientation preserving homeomorphism $h$ of $(S^3,\Gamma)$ inducing $\varphi$ if and only if all vertices are in $r$-cycles except for the fixed vertices and exceptional cycles explicitly mentioned below (up to interchanging $V$ and $W$):

\begin{enumerate} 
\item There are no fixed vertices or exceptional cycles.

\item $V$ contains one or more fixed vertices.

\item $V$ and $W$ each contain at most 2 fixed vertices.

\item $j|r$ and $V$ contains some $j$-cycles.

\item $r=\mathrm{lcm}(j,k)$, and $V$ contains some $j$-cycles and $k$-cycles.

\item $r=\mathrm{lcm}(j,k)$, and $V$ contains some $j$-cycles and $W$ contains some $k$-cycles.

\item $V$ and $W$ each contain one 2-cycle.

\item $\frac{r}{2}$ is odd, $V$ and $W$ each contain one 2-cycle, and $V$ contains some $\frac{r}{2}$-cycles.

\item $\varphi(V)=W$ and $V\cup W$ contains one 4-cycle.

\end{enumerate}
\medskip

\noindent There is an embedding $\Gamma$ of $K_{n,m}$ in $S^3$ with an orientation reversing homeomorphism $h$ of $(S^3,\Gamma)$ inducing $\varphi$ if and only if $r$ is even and all vertices are in $r$-cycles except for the fixed vertices and exceptional cycles explicitly mentioned below (up to interchanging $V$ and $W$):

\begin{enumerate}\setcounter{enumi}{9}

\item  $\f(V) = V$ and there are no fixed vertices or exceptional cycles.

\item  $r=2$ and all vertices of $V$ and at most 2 vertices of $W$ are fixed.

\item  $V$ contains at most 2 fixed vertices, and one of the following is true.

\begin{enumerate}
\item $W$ contains one 2-cycle.

\item $V$ contains some 2-cycles.

\item  $V$ may contain some 2-cycles, $\frac{r}{2}$ is odd, and all vertices of $W$ are in $\frac{r}{2}$ cycles.

\item $W$ contains at most one 2-cycle, $\frac{r}{2}$ is odd, and all non-fixed vertices of $V$ are in $\frac{r}{2}$ cycles.
\end{enumerate}

\item $4 | r$, $\f(V)=W$ and $V\cup W$ contains at most two 2-cycles.

\end{enumerate}
\end{class}

\bigskip

We will begin by proving the necessity of the conditions in the Classification Theorem, and then provide constructions to show that each of the cases listed can actually occur.
\bigskip

\section{Necessity of the Conditions}

The following result allows us to focus our attention on {\em finite order} homeomorphisms of embeddings of our graphs in $S^3$.

\begin{autothm} \cite{Flapan:1995}
Let $\f$ be an automorphism of a 3-connected graph which is induced by a homeomorphism $f$ of $(S^3,\Gamma_1)$ for some embedding $\Gamma_1$ of the graph in $S^3$.  Then $\f$ is induced by a finite order homeomorphism $h$ of $(S^3,\Gamma_2)$ for some possibly different embedding $\Gamma_2$ of $\Gamma_1$ in $S^3$.  Furthermore, $h$ is orientation reversing if and only if $f$ is orientation reversing.
\end{autothm}

 If $n, m\leq 2$, then it is easy to see that all of the automorphisms of $K_{n,m}$ can be induced by a homeomorphism of some embedding of $K_{n,m}$ in $S^3$.  Thus for the remainder of the paper, we focus on $K_{n,m}$ with $n, m >2$.   In this case, $K_{n,m}$ is 3-connected, and hence we can apply the Automorphism Theorem.  It follows that any automorphism which is induced by a homeomorphism for some embedding of $K_{n,m}$ is induced by a finite order homeomorphism for some (possibly different) embedding of the graph.  Thus we begin with some results about finite order homeomorphisms of $S^3$.  The following is a special case of a well-known result of P. A. Smith.

\begin{smith} \cite{Smith:1939}
Let $h$ be a non-trivial finite order homeomorphism of $S^3$.  If $h$ is orientation preserving, then $\fix(h)$ is either the empty set or is homeomorphic to $S^1$.  If $h$ is orientation reversing, then $\fix(h)$ is homeomorphic to either $S^0$ or $S^2$.
\end{smith}
\medskip

\begin{lemma}
\label{topfact}
Let $h$ be a finite order homeomorphism of $S^3$ which is fixed point free.  Then there are at most two circles which are the fixed point set of some power of $h$ less than $r$.
\end{lemma}

\begin{proof}  Suppose that some power of $h$ has non-empty fixed point set.  Then by Thurston's Orbifold Theorem \cite{BLP}, $h$ is conjugate to an orientation preserving isometry $g$ of $S^3$.  Now $g$ can be extended to an orientation-preserving isometry $\hat{g}$ of $\R^4$ fixing the origin; i.e. an element of SO(4).  Every element of SO(4) is the composition of two rotations about perpendicular planes in $\R^4$ (see \cite{duv}, for example), and it is easy to show that these are the only planes fixed by any power of $\hat{g}$.  It follows that there are at most two circles in $S^3$ which are the fixed point set of some power of $g$.  Finally, since $h$ is conjugate to $g$, there are at most two circles which are the fixed point set of some power of $h$.  \end{proof}

\medskip

\begin{lemma}
\label{emptyorder}
Let $h$ be an order $r$ homeomorphism of $S^3$ which is fixed point free such that for some minimal $k < j<r$, $A=\fix(h^k)$ and $B=\fix(h^j)$ are distinct circles.  Then the following are true:
\begin{enumerate}
\item  $A \cap B = \emptyset$.

\item $h(A)=A$ and $h(B)=B$.

\item The points in $A$, $B$, and $S^3-(A \cup B)$ are in $k$, $j$, and $r$-cycles respectively.

\item $r= \lcm (k,j)$.
\end{enumerate}
\end{lemma}

\begin{proof}
Suppose, for the sake of contradiction, that there exists some $x \in A \cap B$.  Then $x$ is fixed by both $h^k$ and $h^j$. Since $k$ is the smallest power of $h$ with a nonempty fixed point set, we have that $k|j$, and so $A \subseteq B$.  But this is impossible since $A$ and $B$ are distinct circles. Thus condition (1) holds.  

Let $x\in A$, then $h^k(x)=x$.  Hence $h^k(h(x))=h(h^k(x))=h(x)$.  So $h(x)$ is fixed by $h^k$.  It follows that $h(x)\in A$.  Thus condition (2) holds.

Observe that $k$ is the smallest power of $h$ that fixes any point of $S^3$.  Thus all of the points in $A$ have order $k$ under $h$.  Similarly, $j$ is the smallest power of $h$ that fixes any point of $S^3 - A$, and hence of $B$.  Thus all of the points in $B$ have order $j$ under $h$.   Finally, it follows from Lemma \ref{topfact} that every point of $S^3-(A \cup B)$ has order $r$ under $h$, and hence condition (3) holds.

Observe that $h^{\lcm(k,j)}$ fixes $A\cup B$ and hence must be the identity.  Thus $r$ divides $\lcm(k,j)$.  But we also know that $k$ and $j$ both divide $r$ since $\fix(h^k)$ and $\fix(h^j)$ are non-empty.  So $\lcm(k,j)$ divides $r$.  It follows that condition (4) holds.\end{proof}
\medskip

Now we consider complete bipartite graphs.  Throughout the paper, we will use $V$ and $W$ to denote the vertex sets of a complete bipartite graph $K_{n,m}$. We begin with the following well known result about automorphisms of complete bipartite graphs.

\medskip

\begin{fact}
\label{auts}
Let $\f $ be a permutation of the vertices of $K_{n,m}$.  Then $\f$ is an automorphism of $K_{n,m}$ if and only if $\f$ either interchanges $V$ and $W$ or setwise fixes each of $V$ and $W$.
\end{fact}

\medskip
 We will use the above Lemmas to obtain necessary conditions on finite order homeomorphisms of complete bipartite graphs embedded in $S^3$.

\begin{lemma}
\label{S2badedges} 
Let $n, m>2$.  Suppose that $\Gamma$ is an embedding of $K_{n,m}$ in $S^3$ with a finite-order homeomorphism $h$ of $(S^3,\Gamma)$.  If $\fix(h)\cong S^2$, then $V$ or $W$ is entirely contained in $\fix(h)$. 
\end{lemma}

\begin{proof}    Since $\fix(h)\cong S^2$, it follows from Smith Theory that $h$ is orientation reversing, and hence interchanges the components
$A$ and $B$ of $S^3-\fix(h)$.  Suppose for the sake of contradiction that neither $V$ nor $W$ is entirely contained in $\fix(h)$.  

 Since neither $V$ nor $W$ is contained in $\fix(h)$ and $A$ and $B$ are interchanged by $h$, there are vertices $v$ of $V$ and $w$ of $W$ such that $v\in A$ and $w\in B$.  Now the edge $ \overline{vw}$ intersects $\fix(h)$, and hence $h$ fixes a point of $\overline{vw}$.  Thus $h$ interchanges $w$ and $v$.  From the above Fact we know that $h(V)=V$ or $h(V)=W$.  Hence we must have $h(V) = W$.  It follows that no vertices of the graph are fixed.  Since $n >2$, without loss of generality, there is a pair of distinct vertices $v_1$ and $v_2$ of $V$ which are contained in $A$.  Then some $h(v_2)\in B$ and $h(v_2)$ is in $W$.  However, the edge $\overline{v_1h(v_2)}$ intersects $\fix(h)$ and hence, $h(v_1) = h(v_2)$.  But this is impossible since $h$ is a bijection.  Thus $\fix(h)$ contains one of the vertex sets.
\end{proof}
\medskip

\begin{3lemma}
\label{3fv}
Let $m, n>2$.  Suppose that $\Gamma$ is an embedding of $K_{n,m}$ in $S^3$ with an order $r$ homeomorphism $h$ of $(S^3,\Gamma)$ which fixes at least one vertex if $h$ is orientation preserving and at least 3 vertices if $h$ is orientation reversing.  Then all non-fixed vertices are in $r$-cycles and one of the following holds (up to interchanging $V$ and $W$):
\begin{enumerate}

\item If $h$ is orientation preserving, then either $h$ fixes no vertices of $W$, or $h$ fixes at most 2 vertices of each of $V$ and $W$.

\item If $h$ is orientation reversing, then $h$ has order 2, and $h$ fixes all vertices of $V$ and at most 2 vertices of $W$.
	
\end{enumerate}
\end{3lemma}

\begin{proof}
First suppose that $h$ is orientation preserving.  If $h$ fixes at least 3 vertices of $V$, then $h$ cannot fix any vertices of $W$, since otherwise $\fix(h)\cong S^1$ would contain a $K_{3,1}$ graph.  Thus condition (1) is satisfied.  

Next suppose that $h$ is orientation reversing.  Since at least 3 vertices are fixed by $h$, by Smith Theory $\fix(h)\cong S^2$.  Hence $h$ has order 2.  Also, by Lemma \ref{S2badedges}, without loss of generality $\fix(h)$ contains all of $V$, which includes at least 3 vertices.  If $\fix(h)$ also contained more than 2 vertices of $W$, then $\fix(h)\cong S^2$ would contain the nonplanar graph $K_{3,3}$.  Thus condition (2) is satisfied.

In either case, by Smith Theory all non-fixed vertices are in $r$-cycles. \end{proof}
\medskip

\begin{12lemma}
\label{2fv}
Let $m, n>2$.  Suppose that $\Gamma$ is an embedding of $K_{n,m}$ in $S^3$ with an order $r$ orientation reversing homeomorphism $h$ of $(S^3,\Gamma)$ which fixes at most 2 vertices.  Then the fixed vertices are in a single vertex set $V$, and one of the following holds, with all remaining non-fixed vertices in $r$-cycles:\begin{enumerate}

	\item There are no fixed vertices or exceptional cycles, and $h(V) = V$.
	
	\item $W$ contains one 2-cycle.

	\item  $V$ contains some 2-cycles.
	
	 \item $V\cup W$ contains at most two 2-cycles, $\frac{r}{2}$ is even, and $h(V) = W$.
	
	\item  $V$ may contain some 2-cycles, $\frac{r}{2}$ is odd, and all vertices of $W$ are in $\frac{r}{2}$-cycles.
	
	\item $W$ contains at most one 2-cycle, $\frac{r}{2}$ is odd, and all non-fixed vertices of $V$ are in $\frac{r}{2}$-cycles.
	
\end{enumerate}
\end{12lemma}

\begin{proof}  Observe that since $h$ is orientation reversing, $r$ must be even.  We will observe that if $\frac{r}{2}$ is odd, then $h(V) = V$.  Towards contradiction, suppose $\frac{r}{2}$ is odd, and $h(V) = W$.  Then $h^{\frac{r}{2}}$ is also orientation-reversing, interchanges $V$ and $W$ and divides $V \cup W$ into 2-cycles $(v_iw_i)$.  But then $h^{\frac{r}{2}}$ must fix a point on each edge $\overline{v_iw_i}$; since $m, n > 2$, this means $\fix(h^{\frac{r}{2}}) = S^2$.  But, by Lemma \ref{S2badedges}, this contradicts the fact that $h^{\frac{r}{2}}$ does not fix any vertices.  Hence $h(V) = V$.  In particular, if $r = 2$ then $h(V) = V$, and conditions (1) and/or (3) follows trivially.

So we assume that $r>2$.  Since $n,m \geq 3$ and no more than 2 vertices are fixed by $h$, both $V$ and $W$ must contain non-fixed vertices.  Thus by Lemma \ref{S2badedges}, $\fix(h)$ cannot be homeomorphic to $S^2$.  Now it follows from Smith Theory that $h$ has precisely 2 fixed points.  Thus $h$ cannot fix one vertex from each vertex set, since that would force $h$ to fix every point on the edge between the fixed vertices.  So without loss of generality, all fixed vertices are in $V$.  

Since $h^2$ is orientation preserving and has non-empty fixed point set, $\fix(h^2)\cong S^1$.  By applying Smith Theory to $h^2$ we see that all points not fixed by $h^2$ are in $\frac{r}{2}$-cycles under $h^2$, and hence are either in $\frac{r}{2}$-cycles or $r$-cycles under $h$. So all non-fixed vertices of $\Gamma$ are in 2-cycles, $\frac{r}{2}$-cycles, or $r$-cycles under $h$.  

Suppose at least one vertex of $V$ is fixed and some vertex $w\in W$ is in a 2-cycle.  Then $h^2$ fixes these 2 vertices of $W$ together with any vertices of $V$ which are fixed by $h$.  If $W$ contained a second 2-cycle then $\fix(h^2)$ would contain a $K_{1,4}$ graph, and if $V$ contained a 2-cycle then $\fix(h^2)$ would contain a $K_{3,2}$ graph.  Both cases are impossible since $\fix(h^2)\cong S^1$.  Hence the 2-cycle in $W$ is the only 2-cycle in the graph.  So, if there is a fixed vertex and no vertices of $\Gamma$ are in $\frac{r}{2}$-cycles (or $r = 4$), then either condition (2) or (3) is satisfied depending on whether or not $W$ contains a 2-cycle.

Suppose there are no fixed vertices in $\Gamma$.  Then the vertex sets $V$ and $W$ are interchangeable.  As above, if $\fix(h^2)$ contains more than 2 vertices of $V$, then it cannot contain any vertices of $W$ since $\fix(h^2)\cong S^1$.  Thus either there are at most two 2-cycles in $V\cup W$ and these are the only 2-cycles in $\Gamma$, or all of the 2-cycles of $\Gamma$ are in $V$.   If there are 2-cycles $(v_1v_2)$ and $(w_1w_2)$ then $\fix(h^2)$ must be exactly these four vertices and the edges between them.  But this set contains the two points of $\fix(h)$; since $h$ does not fix a vertex it must fix a point on an edge, and hence must interchange the endpoints of this edge.  But then $h(V) = W$, which contradicts the assumption that $h(v_1) = v_2$.

So, if there are 2-cycles involving both $V$ and $W$, we must have that $h(V) = W$, and hence $\frac{r}{2}$ must be even.  Thus either there are at most two 2-cycles in $V\cup W$, $\frac{r}{2}$ is even and $h(V) = W$, or all of the 2-cycles of $\Gamma$ are in $V$ and $h(V) = V$.  In particular, if $r=4$ or no vertices of $\Gamma$ are in $\frac{r}{2}$-cycles, then one of conditions (1), (3) or (4) is satisfied.

So we can assume that some vertex of $\Gamma$ is in an $\frac{r}{2}$-cycle and $r\not =4$.  Then the vertices in the $\frac{r}{2}$-cycle are not fixed by $h^2$.  But if $\frac{r}{2}$ is even, then $\fix(h^2) \subseteq \fix(h^{\frac{r}{2}}) = S^1$, so the fixed point sets are the same.  Hence $\frac{r}{2}$ must be odd and
$r\geq 6$.   Thus $h^{\frac{r}{2}}$ fixes at least 3 vertices from a $\frac{r}{2}$-cycle.  Since $h^{\frac{r}{2}}$ is orientation reversing, this implies that $\fix(h^{\frac{r}{2}})\cong S^2$.  Now by Lemma \ref{S2badedges}, $\fix(h^{\frac{r}{2}})$ contains all of the vertices in one of the vertex sets.  Let $X$ 
denote the vertex set contained in $\fix(h^{\frac{r}{2}})$ and let $Y$ denote the other vertex set.  Hence all of the non-fixed vertices of $X$ are in $\frac{r}{2}$-cycles.  Suppose that some vertex $y\in Y$ is also in a $\frac{r}{2}$-cycle.  Then $y$, $h(y)$, and $h^2(y)$ are distinct vertices in $\fix(h^{\frac{r}{2}})$.  Since $X$ contains at least 3 vertices, this would imply that $\fix(h^{\frac{r}{2}})$ contains a $K_{3,3}$ graph.  As this is impossible, no vertex in $Y$ is in an $\frac{r}{2}$-cycle.  In particular, $V$ and $W$ cannot both contain $\frac{r}{2}$-cycles.  Now if $W$ is contained in $\fix(h^{\frac{r}{2}})$ then condition (5) is satisfied.  If there are no fixed vertices, then $V$ and $W$ are interchangeable, and hence again condition (5) is satisfied.  If there is at least one fixed vertex and $V$ is contained in $\fix(h^{\frac{r}{2}})$, then condition (6) is satisfied since in this case $W$ contains at most one 2-cycle.  \end{proof}

\medskip

The situation is somewhat more complicated if $h$ is orientation preserving and fixes no vertices of $\Gamma$.  In this case, by Lemma \ref{topfact} one or two distinct circles (but no more) could be fixed by different powers of $h$.  This gives many possibilities depending on whether these circles contain vertices from the sets $V$ and/or $W$.

\begin{0lemma}
\label{0fv}
Let $m, n>2$.  Suppose that $\Gamma$ is an embedding of $K_{n,m}$ in $S^3$ with an order $r$ orientation preserving homeomorphism $h$ of $(S^3,\Gamma)$ which fixes no vertices.   If some power of $h$ less than then $r$ fixes a vertex, then one of the following holds (up to interchanging $V$ and $W$), with all remaining vertices in $r$-cycles:
\begin{enumerate}

\item If there is precisely one circle $\fix(h^j)$ containing a vertex  and $j<r$ is minimal, then one of the following holds.

\begin{enumerate}
\item $V$ may contain some $j$-cycles.

\item $\fix(h^j)$ contains vertices $v_1, v_2 \in V$ and $w_1, w_2 \in W$ such that either $j=2$ and $h$ induces $(v_1v_2)(w_1w_2)$ or $j=4$ and $h$ induces $(v_1w_1v_2w_2)$.

\end{enumerate}

\item If there are two circles $\fix(h^j)$ and $\fix(h^k)$ containing vertices and $j, k<r$ are minimal, then one of the following holds.

\begin{enumerate}
\item  $V$ contains some $j$-cycles and $k$-cycles.

\item $V$ contains some $j$-cycles and $W$ contains some $k$-cycles.

\item $j=2$, $\fix(h^2)$ contains $v_1, v_2 \in V$ and $w_1, w_2 \in W$ such that $h$ induces $(v_1v_2)(w_1w_2)$, and $V$ contains some $\frac{r}{2}$-cycles with $\frac{r}{2}$ odd.

\end{enumerate}
\end{enumerate}
\end{0lemma}

\begin{proof}
By our hypotheses, some power of $h$ less than $r$ fixes a vertex, though $h$ itself fixes no vertices.  Since $h$ is orientation preserving it follows from Smith Theory that $h$ is fixed point free.  

First suppose that there is a circle $A=\fix(h^j)$ which contains $v_1\in V$ and $w_1\in W$ and $j$ is minimal.  Since $h$ is fixed point free and $h(A)=A$, we know that $h$ rotates the circle $A$.  Thus $A$ contains the same number of vertices from each of the vertex sets $V$ and $W$, and since $A$ must also contain the edges between vertices this number must be at least 2 (or $h$ would not send edges to edges).  Thus $A$ must consist of vertices $v_1$, $w_1$, $v_2$, $w_2$ together with the edges between them, and $j = 2$ or $4$.  If no vertex is fixed by any other power of $h$ less than $r$, then condition (1b) is satisfied.  Suppose that some vertex not on $A$ is fixed by $h^k$ with $k$ minimal such that $j < k <r$.  Since $h^k$ does not fix any point of $A$, $j$ cannot divide $k$.  In particular, $B=\fix(h^k)$ cannot contain vertices from both $V$ and $W$ (since then $k = 2$ or $4$ by the same argument, and $j$ would divide $k$).  Thus without loss of generality, $B$ only contains vertices of $V$.  Since $h(B)=B$, we must have $h(V)=V$.  Thus $h$ must induce $(v_1v_2)(w_1w_2)$. Now it follows from Smith Theory that $k$ is odd and from Lemma \ref{emptyorder} that $k=\frac{r}{2}$.  Hence condition (2c) is satisfied.

Thus we assume that no power of $h$ less than $r$ simultaneously fixes vertices from each of $V$ and $W$.  Now by Lemmas \ref{topfact} and \ref{emptyorder}, it is easy to check that one of conditions (1a), (2a), or (2b) is satisfied.\end{proof}

\medskip

%%%%%%%%%%%%%%%%%%%%%%%%%%%%%%%%%%%%%%%%%%%%%%%%%%%%%

%Nicole's thesis:  Realizability for K_{n,m}

\section{Realizing Automorphisms of Complete Bipartite Graphs}

We will provide constructions to show each case of the Classification Theorem is possible.  Our constructions proceed by first embedding the vertices of $K_{n,m}$ so that there is an isometry of $S^3$ that acts on them as desired, and then embedding the edges using the following Edge Embedding Lemma, so that the resulting embedding of $K_{n,m}$ is setwise invariant under $g$, and $g$ induces $\f$ on $K_{n,m}$.  Recall that a graph $H$ is a {\em subdivision} of a graph $G$ if $H$ is the result of adding distinct vertices to the interiors of the edges of $G$.

\begin{eel}\label{eelemma}
\cite{Flapan:2010}
Let $V$ and $W$ denote the vertex sets of $K_{n,m}$ and let $\g$ be a subdivision of $K_{n, m}$ constructed by adding vertices $Z$.  Assume the vertices of $\g$ are embedded in $S^3$ so that an isometry $g$ of $S^3$ of order $r$ induces a faithful action on $\g$.  Let $Y$ denote the union of the fixed point sets of all the $g^i$, for $i < r$.  Suppose the following hypotheses hold for adjacent pairs of vertices in $V \cup W \cup Z$:
\begin{enumerate}

\item
If a pair is pointwise fixed by $g^i$ and $g^j$ for some $i,j<r$, then $\fix(g^i)=\fix(g^j)$.

\item No pair is interchanged by any $g^i$.

\item
Any pair that is pointwise fixed by some $g^i$ with $i<r$ bounds an arc in $\fix(g^i)$ whose interior is disjoint from $V\cup W \cup Z \cup (Y - \fix(g^i))$.

\item
Every pair is contained in the closure of a single component of $S^3 - Y$

\end{enumerate}

Then there is an embedding of the edges of $\g$ in $S^3$ such that the resulting embedding of $\g$ is setwise invariant under $g$.  
\end{eel}

Note that if $g$ is an isometry, then $Y$ only separates $S^3$ if $\fix(g^i) = S^2$ for some $i$.  So, by Lemma \ref{S2badedges}, condition (4) of the Edge Embedding Lemma is equivalent to saying that if $\fix(g^i)=S^2$ then either $V\subseteq \fix(g^i)$ or $W\subseteq \fix(g^i)$.
\medskip

%%%%%%%%%%%%%%%%%%%%%%%%%%%%%%%%%%%%%%%%%%%%%%%%%%%%%%%%%%%%%%%%%%%%%%%%%%
%%%%%%%%%%%%%%%%%%%%%%%%%%%%%%%%%%%%%%%%%%%%%%%%%%%%%%%%%%%%%%%%%%%%%%%%%%

We will first consider automorphisms $\f$ of $K_{n,m}$ which can be realized by {\it orientation-preserving} isometries of $S^3$.  The roles of $V$ and $W$ can be reversed in all the following lemmas.  Notice that the conditions of each lemma are numbered to match the corresponding conditions in the Classification Theorem.  We first show that the first three cases of the Classification Theorem are realizable by rotations in $S^3$, as long as the automorphism fixes $V$ and $W$ setwise.

%%%%%%%%%%%%%%%

\begin{rotation}\label{r-cycles}
Let $m, n>2$ and let $\f$ be an order $r$ automorphism of $K_{n,m}$ with vertex sets $V$ and $W$. Suppose that $\f(V) = V$ and all vertices are in $r$-cycles except for the fixed vertices and exceptional cycles explicitly mentioned below (up to interchanging $V$ and $W$): 
\begin{enumerate} 
\item There are no fixed vertices or exceptional cycles.

\item $V$ contains one or more fixed vertices.

\item $V$ and $W$ each contain at most 2 fixed vertices.
\end{enumerate}

\noindent Then $\f$ is realizable by a rotation of $S^3$ of order $r$.

\end{rotation}

\begin{proof}
Let $g$ be a rotation of $S^3$ by $\frac{2\pi}{r}$ around a circle $X$.  Then the order of $g$ is $r$.  Embed all the fixed points of $\f$ on $X$.  In the case where $V$ and $W$ each contain 2 fixed vertices, alternate vertices from $V$ and $W$ around $X$.  Embed the remaining vertices in pairwise disjoint $r$-cycles of $g$ in $S^3-X$.  Thus $g$ induces $\f$ on $V\cup W$.

We now check that the conditions of the Edge Embedding Lemma are satisfied.  Since $g(V)=V$ and $g(W)=W$, condition (2) is satisfied.  Notice that for every $i<r$ we have that $\fix(g^i)=X$, and $X$ does not separate $S^3$, so conditions $(1)$ and $(4)$ are satisfied.  Furthermore, if $\f$ fixes either $1$ or $2$ vertices in each of $V$ and $W$, there is a collection of arcs in $X$ from each of the fixed points of $V$ to each of the fixed points of $W$ such that the interior of each arc is disjoint from $V\cup W$.  Hence condition $(3)$ is met.  Thus by the Edge Embedding Lemma, there is an embedding of $K_{n,m}$ such that $\f$ is realized by $g$.  
\end{proof}

%%%%%%%%%%%%%%%%%%%%%%%%%%%%%%%%%%%%%%%%%%%%%%%%%%%%%%%%%%%%%%%%%%%%%%%%%%
\bigskip

The next 6 cases of the Classification Theorem are realizable by glide rotations in $S^3$, along with the automorphisms in case (1) that interchange the vertex sets $V$ and $W$.  Recall that a {\em glide rotation} is the composition of two rotations about linked circles such that each rotation fixes the axis of the other rotation setwise.

\begin{glide}\label{kjr-cycles}
Let $m, n>2$ and let $\f$ be an order $r$ automorphism of $K_{n,m}$ with vertex sets $V$ and $W$. Suppose that there are no fixed vertices, and all vertices are in $r$-cycles except for the exceptional cycles explicitly mentioned below (up to interchanging $V$ and $W$): 
\begin{enumerate}

	\item $\f(V) = W$ and there are no exceptional cycles.
	
\setcounter{enumi}{3}

	\item $j|r$ and $V$ contains some $j$-cycles.

	\item $r=\mathrm{lcm}(j,k)$, and $V$ contains some $j$-cycles and $k$-cycles.

	\item $r=\mathrm{lcm}(j,k)$, and $V$ contains some $j$-cycles and $W$ contains some $k$-cycles.

	\item $V$ and $W$ each contain one 2-cycle.

	\item $\frac{r}{2}$ is odd, $V$ and $W$ each contain one 2-cycle, and $V$ contains some $\frac{r}{2}$-cycles.

	\item $\f(V)=W$ and $V\cup W$ contains one 4-cycle.

\end{enumerate}

\noindent Then $\f$ is realizable by a glide rotation of $S^3$ of order $r$.
\end{glide}

\begin{proof}
Let $X$ and $Y$ be geodesic circles in $S^3$ which are the intersections of $S^3$ with perpendicular planes through the origin in $\R^4$.  We define a glide rotation $g$ in each case as follows (we split case (1) into two parts, depending on whether $\frac{r}{2}$ is odd or even; since $\f(V) = W$, we know $r$ must be even):

\begin{enumerate}
	\item[(1a)] ($\frac{r}{2}$ is odd) $g$ is the composition of a rotation by $\frac{4\pi}{r}$ around $X$ and a rotation by $\frac{2\pi}{r}$ around $Y$.
	\item[(1b)] ($\frac{r}{2}$ is even) $g$ is the composition of a rotation by $\frac{\pi}{2}$ around $X$ and a rotation by $\frac{2\pi}{r}$ around $Y$.
\setcounter{enumi}{3}
	\item $g$ is the composition of a rotation by $\frac{2\pi}{j}$ around $X$ and a rotation by $\frac{2\pi}{r}$ around $Y$.
	\item $g$ is the composition of a rotation by $\frac{2\pi}{j}$ around $X$ and a rotation by $\frac{2\pi}{k}$ around $Y$.
	\item same as case (5).
	\item $g$ is the composition of a rotation by $\pi$ around $X$ and a rotation by $\frac{2\pi}{r}$ around $Y$.
	\item $g$ is the composition of a rotation by $\pi$ around $X$ and a rotation by $\frac{4\pi}{r}$ around $Y$.
	\item same as case (1b).
\end{enumerate}

Observe that in case (7) $r$ is even since there is a 2-cycle.  In case (8) $\frac{r}{2}$ is odd by assumption, and in case (9) $r$ is divisible by 4 since there is a 4-cycle (so $\frac{r}{2}$ is even).  So in each case, the order of $g$ is $r$.

We begin by embedding the vertices of $K_{n,m}$ in each case.  In cases (4), (5) and (6), $g\vert Y$ has order $j$, and we embed all of the $j$-cycles of $V$ under $\f$ in $Y$ as pairwise disjoint $j$-cycles of $g$.  In cases (7) and (8), $g\vert Y$ has order $2$, andwe embed the $2$-cycles of $V$ and $W$ in $Y$ as alternating pairwise disjoint $2$-cycles of $g$.  In case (9), $g\vert Y$ has order $4$, and we embed the 4-cycle in $Y$ with $v$'s and $w$'s alternating.  In cases (5) and (6), $g \vert X$ has order $k$, and we similarly embed all of the $k$-cycles of $V$ or $W$ under $\f$ in $X$.  In case (8), $g \vert X$ has order $\frac{r}{2}$, and we embed all of the $\frac{r}{2}$-cycles of $V$ under $\f$ in $X$.  Finally, in all the cases we embed the $r$-cycles of $V$ and $W$ in $S^3 - (X \cup Y)$.  Thus $g$ induces $\f$ on $V\cup W$.

\medskip

We will now use the Edge Embedding Lemma to embed the edges of $K_{n,m}$.  We will first consider case (1a), which is a bit more complex than the others.  In this case we need to first embed the midpoints of all the edges that are inverted by any $g^i$ for $i<r$.  Since all the cycles are length $r$, this is only possible if $i=\frac{r}{2}$.  Furthermore, since $\frac{r}{2}$ is odd, $g^\frac{r}{2}$ inverts $n$ edges of $K_{n,n}$ (since $\f(V) = W$, we have $n=m$ in this case).  For each edge $\overline{vw}$ which is inverted by $\f ^\frac{r}{2}$ we add a vertex $z_{vw}$ at the midpoint of the edge of the abstract graph $K_{n,n}$.  Denote this set of $n$ vertices by $Z$.  Observe that the vertices of $Z$ are in $\frac{r}{2}$-cycles under $\f$.  Then we define $H=K_{n,n}\cup Z$, so $H$ is a subdivision of $K_{n,n}$.  Now $g\vert Y$ has order $\frac{r}{2}$.  We embed the $\frac{r}{2}$-cycles of vertices $z_{vw}$ under $\f$ as a $\frac{r}{2}$-cycle of $g$ in $Y$.  Thus $g$ induces $\f$ on $V\cup W\cup Z$. 

We now check that the conditions of the Edge Embedding Lemma are satisfied for $H$.  If any pair of adjacent vertices of $H$ is fixed by $g^i$, then $g^i$ must fix a vertex of $V \cup W$, but none of these vertices are embedded in $X \cup Y$.  So conditions (1) and (3) are trivially satisfied.  Since $\fix(g^i) \neq S^2$ for all $i$, condition (4) is satisfied.  For every pair of vertices $v\in V$ and $w\in W$ such that $g^\frac{r}{2}(v)=w$, there exists a vertex $z_{vw}\in Z$ such that $z_{vw}\in\fix(g^\frac{r}{2})=X$, so $v$ and $w$ are not adjacent in $H$.  Hence condition $(2)$ is satisfied.  Thus by the Edge Embedding Lemma, there is an embedding of $H$ that is setwise invariant under $g$.  Finally, we delete the embedded midpoint vertices of our embedding of $H$ to obtain an embedding of $K_{n,n}$ such that $\f$ is realized by $g$.

\medskip

We now check that the conditions of the Edge Embedding Lemma are satisfied for the other cases.  In cases (4)--(8) we have that $g(V)=V$ and $g(W)=W$, hence condition (2) is satisfied.  For cases (1b) and (9) we observe that $\f^i$ inverts an edge of $K_{n,n}$ if and only if $i$ is odd and $\f$ has a vertex cycle of length $2i$.  But the only vertex cycles in these case are length 4 or $r$, and $\frac{r}{2}$ is even.  So none of the edges are inverted, and condition (2) is again satisfied.  Since no $g^i$ with $i<r$ pointwise fixes $X\cup Y$, if two points are fixed by $g^i$ then they are either both in $X$ or both in $Y$.  The only cases when vertices of both $V$ and $W$ are embedded in the same circle is when both $V$ and $W$ have a single 2-cycle embedded in $Y$, or when a 4-cycle is embedded in $Y$.  If $g^i$ and $g^j$ both fix such a pair, for $i,j < r$, then $\fix(g^i) = \fix(g^j) = Y$, so condition (1) is satisfied.  Also, since the 4 vertices in each case alternate between $v$'s and $w$'s, condition (3) is satisfied.

Otherwise, vertices of $W$ and $V$ are not both embedded in the same circle.  Hence $g^i$ for $i<r$ does not simultaneously fix both a $v\in V$ and a $w\in W$, and conditions $(1)$ and $(3)$ follow trivially.  Since $\fix(g^i)\neq S^2$ for every $i<r$, condition $(4)$ is satisfied.  Thus by the Edge Embedding Lemma, there is an embedding of $K_{n,m}$ such that $\f$ is realized by $g$.
\end{proof}

\bigskip
%%%%%%%%%%%%%%%%%%%%%%%%%%%%%%%%%%%%%%%%%%%%%%%%%%%%%%%%%%%%%%%%%%%%%%%%%
We now consider the automorphisms that can be realized by {\em orientation-reversing} isometries of $S^3$.  Case (11) of the Classification Theorem can be realized by a reflection.

\begin{reflection}\label{order2}
Let $m, n>2$ and let $\f$ be an automorphism of $K_{n,m}$ of order $2$ that fixes all of the vertices of $V$ and at most $2$ vertices of $W$, with the remaining vertices of $W$ partitioned into $2$-cycles.  Then $\f$ is realizable by a reflection of $S^3$.
\end{reflection}

\begin{proof}
Let $g$ be a reflection of $S^3$ through a sphere $S$.  Embed all the fixed vertices of $\f$ on $S$, and embed the $2$-cycles of $\f$ in $S^3-S$ as pairwise disjoint $2$-cycles of $g$.  Thus $g$ induces $\f$ on $V\cup W$.

We now check that the conditions of the Edge Embedding Lemma are satisfied by the embedded vertices $V\cup W$ of $K_{n,m}$.  By construction we have that $g(V)=V$ and $g(W)=W$, hence condition (2) is satisfied.  Since the order of $g$ is $2$, condition $(1)$ of the Edge Embedding Lemma is met.  Since $K_{2,n}$ is a planar graph, there is a collection of disjoint arcs in $S$ connecting the vertices of $V$ to the (at most 2) fixed vertices of $W$ in $S$, hence condition $(3)$ is satisfied.  Furthermore, since $V\subseteq \fix(g)$, condition $(4)$ is satisfied.  Thus by the Edge Embedding Lemma, there is an embedding of $K_{n,m}$ such that $\f$ is realized by $g$.  \end{proof}
\bigskip
%%%%%%%%%%%%%%%%%%%%%%%%%%%%%%%%%%%%%%%%%%%%%%%%%%%%%%%%%%%%%%%%%%

Cases (10), (12) and (13) of the Classification Theorem can be realized by improper rotations.  An {\em improper rotation} of $S^3$ is the composition of a reflection through a geodesic 2-sphere $S$ and a rotation about a geodesic circle $X$ that intersects $S$ perpendicularly in two points.

\begin{improper}\label{r/2V}
Let $m, n>2$ and let $\f$ be an order $r$ (where $r$ is even) automorphism of $K_{n,m}$ with vertex sets $V$ and $W$. Suppose that all vertices are in $r$-cycles except for the fixed vertices and exceptional cycles explicitly mentioned below (up to interchanging $V$ and $W$):

\begin{enumerate}\setcounter{enumi}{9}

\item $\f(V) = V$ and there are no fixed vertices or exceptional cycles.

\setcounter{enumi}{11}

\item  $\f(V) = V$, $V$ contains at most 2 fixed vertices, and one of the following is true.

\begin{enumerate}
\item $W$ contains one 2-cycle.

\item $V$ contains some 2-cycles.

\item  $V$ may contain some 2-cycles, $\frac{r}{2}$ is odd, and all vertices of $W$ are in $\frac{r}{2}$ cycles.

\item $W$ contains at most one 2-cycle, $\frac{r}{2}$ is odd, and all non-fixed vertices of $V$ are in $\frac{r}{2}$ cycles.
\end{enumerate}

\item $4 | r$, $\f(V)=W$ and $V\cup W$ contains at most two 2-cycles.

\end{enumerate}

\noindent Then $\f$ is realizable by an improper rotation of $S^3$ of order $r$.
\end{improper}

\begin{proof}
Let $S$ be a geodesic sphere in $S^3$ and $X$ be a geodesic circle which intersects $S$ perpendicularly at exactly two points.  In cases (10), (12a), (12b) and (13), we define an improper rotation $g$ as the composition of a reflection in $S$ with a rotation by $\frac{2\pi}{r}$ around $X$.  In cases (12c) and (12d) (where $\frac{r}{2}$ is odd), we define $g$ as the composition of a reflection in $S$ with a rotation by $\frac{4\pi}{r}$ around $X$.  In each case, $g$ has order $r$.  We will now embed the vertices of $K_{n,m}$.

\medskip

Let $F = \fix(g) = X \cap S$, so $\vert F \vert = 2$.  In each case, embed the fixed vertices of $\f$ (if any) as points of $F$.  If $r = 2$, then $\f(V) = V$ and we embed the remaining vertices as 2-cycles in $S^3 - (S \cup X)$; in this case, the conditions of the Edge Embedding Lemma are easily verified.  So we will assume $r > 2$.

Observe $g\vert X$ has order $2$.  In each case, we embed any 2-cycles under $\f$ in $X-F$ as pairwise disjoint $2$-cycles of $g$ (in cases (12a) and (12d), the embedded vertices of $W$ will alternate with any vertices of $V$ embedded in $F$); in case (13), we also alternate vertices from $V$ and $W$ around $X$.  Next observe that in cases (12c) and (12d) $g\vert S$ has order $\frac{r}{2}$.  In these cases, we embed the $\frac{r}{2}$-cycles under $\f$ in $S-F$ as pairwise disjoint $\frac{r}{2}$-cycles of $g$.  Finally, we embed the $r$-cycles under $\f$ as pairwise disjoint $r$-cycles of $g$ in $S^3-(S\cup X)$.  Thus $g$ induces $\f$ on $V\cup W$.

We now check that the conditions of the Edge Embedding Lemma are satisfied.  We first consider cases (10) and (12); case (13) is slightly more complex.  In cases (10) and (12) we have that $g(V)=V$ and $g(W)=W$, hence condition (2) is satisfied.  Since no $g^i$ with $i<r$ pointwise fixes $S\cup X$, if $2$ vertices are fixed by $g^i$ then they are both in $S$ or both in $X$.  In each of cases (10) and (12a)--(12d), all the vertices embedded in $X-F$ are from the same vertex set and all the vertices embedded in $S-F$ are from the same vertex set.  So we only need to consider the cases when there are vertices of $V$ embedded in $F$ and vertices of $W$ in either $X-F$ (cases (12a) and (12d)) or $S-F$ (case (12c)).  In cases (12a) and (12d), any $g^i$ (with $i < r$) that fixes an adjacent pair has fixed point set $X$, so condition (1) is satisfied; since there are at most two vertices of $V$ and $W$ embedded in $X$, alternating between $V$ and $W$, condition (3) is also satisfied.  In case (12c), any $g^i$ (with $i < r$) that fixes an adjacent pair has fixed point set $S$, so condition (1) is satisfied; since $K_{2,n}$ is planar there there is a collection of disjoint arcs connecting the vertices of $W$ in $S-F$ to the vertices of $V$ in $F$, so condition (3) is also satisfied.  The only cases when $\fix(g^i) = S$ are cases (12c) and (12d), when $i = \frac{r}{2}$.  In these cases, either $V \subset S$ or $W \subset S$, so condition 4 is satisfied.  Thus by the Edge Embedding Lemma, there is an embedding of $K_{n,m}$ such that $\f$ is realized by $g$.  

\medskip

Finally, we consider case (13).  In this case we need to first embed the midpoints of any edges that are inverted by any $g^i$ for $i<r$.  Since $\frac{r}{2}$ is even, $\f^{\frac{r}{2}}(V) = V$, so this will not happen for any of the vertices in the $r$-cycles; it remains to consider the 2-cycles.  For each 2-cycle $(vw)$ we add a vertex $z_{vw}$ at the midpoint of the edge $\overline{vw}$ in the abstract graph $K_{n,n}$ (since $\f(V) = W$, we have $n = m$).  Denote this set of vertices by $Z$; note that $\vert Z \vert \leq 2$.  Then we define $H=K_{n,n}\cup Z$, so $H$ is a subdivision of $K_{n,n}$.  The vertices of $Z$ are fixed by $\f$, so we embed them as points in $F$ (these points are available, since $\f$ doesn't fix any vertices of $V \cup W$).  Thus $g$ induces $\f$ on $V\cup W\cup Z$. 

We now check that the conditions of the Edge Embedding Lemma are satisfied for $H$.  If any pair of adjacent vertices of $H$ is fixed by $g^i$, then $g^i$ must fix $X$, so condition (1) is satisfied.  Since the vertices of $V$ and $W$ alternate around $X$ (with at most one vertex of $Z$ in between each pair), condition (3) is satisfied.  Since $\fix(g^i) \not S$ for all $i$, condition (4) is satisfied.  For every pair of vertices $v\in V$ and $w\in W$ such that $g(v)=w$, there exists a vertex $z_{vw}\in Z$ such that $z_{vw}\in\fix(g)=F$, so $v$ and $w$ are not adjacent in $H$.  Hence condition $(2)$ is satisfied.  Thus by the Edge Embedding Lemma, there is an embedding of $H$ that is setwise invariant under $g$.  Finally, we delete the embedded midpoint vertices of our embedding of $H$ to obtain an embedding of $K_{n,n}$ such that $\f$ is realized by $g$.
\end{proof}

\section{Conclusion}

\noindent{\sc Proof of the Classification Theorem.}  The proof follows immediately from our lemmas.  Let $\f$ be an automorphism.  If $\vert \fix(\f) \vert \geq 1$, then $\f$ is only realizable by an orientation-preserving homeomorphism if it falls into case (2) or (3), by the Fixed Vertices Lemma.  In these cases, $\f$ is realizable by a rotation.  If $\vert \fix(\f) \vert = 0$, but some power of $\f$ fixes a vertex, then $\f$ is only realizable by an orientation-preserving homeomorphism if it falls into one of cases (4)--(9), by the Orientation Preserving Lemma.  In these cases, $\f$ is realizable by a glide rotation.  If $\vert \fix(\f) \vert = 0$ and no power of $\f$ fixes a vertex, then $\f$ is realizable by either a rotation or a glide rotation, depending on whether $\f(V) = V$.

If $\vert \fix(\f) \vert \geq 3$, then $\f$ is only realizable by an orientation-reversing homeomorphism if it falls into case (11) by the Fixed Vertices Lemma.  In this case, $\f$ is realizable by a reflection.  If $1 \leq \vert \fix(\f) \vert \leq 2$, then $\f$ is only realizable by an orientation-reversing homeomorphism if it falls into case (12), by the Orientation Preserving Lemma.  In this case, $\f$ is realizable by an improper rotation.  If $\vert \fix(\f) \vert = 0$, then $\f$ is only realizable by an orientation-reversing homeomorphism if it falls into case (10), (12) or (13), by the Orientation Preserving Lemma.  In each of these cases, $\f$ is realizable by an improper rotation.  This accounts for all possibilities, and completes the proof.  $\Box$

\medskip

Now that we have determined which automorphisms of $K_{n,m}$ can be induced by a homeomorphism of an embedding in $S^3$, the next step is to explore which {\em groups} of automorphisms can be realized as the group of symmetries of an embedding in $S^3$.  It is known that for every finite subgroup $G$ of $\mathrm{Diff}_+(S^3)$, there is an embedding $\Gamma$ of some complete bipartite graph $K_{n,n}$ such that the group of all automorphisms of $\Gamma$ which are induced by orientation preserving homeomorphisms of $S^3$ is isomorphic to $G$ \cite{TSG1}.  However, this result does not tell us which groups are possible for embeddings of a {\em particular} bipartite graph.

\begin{question}
Given $n$ and $m$, which subgroups of the automorphism group of $K_{n,m}$ are induced by the orientation preserving homeomorphisms of the pair $(S^3, \Gamma)$ for some embedding $\Gamma$ of $K_{n,m}$?
\end{question}

The analogous question for complete graphs has been completely answered \cite{Flapan:2010}.  For bipartite graphs, the question has been studied for $K_{n,n}$ when the subgroup is isomorphic to $A_4$, $S_4$ or $A_5$ \cite{me} and when it is isomorphic to $\Z_n$, $D_n$ or $\Z_n \times \Z_m$ \cite{hmp}, but there is substantial work still to be done.

\small

\end{document}